\documentclass[11pt, twoside]{article}
\usepackage{amssymb, amsmath, amsthm}
\usepackage[left=25mm, right=25mm, bottom=30mm]{geometry}
\usepackage{fancyhdr}
\usepackage{tikz}
\usetikzlibrary{positioning}
\usepackage{titling}
\usepackage{url}
\predate{}
\postdate{}
\usepackage{lipsum}
\newtheorem{theorem}{Theorem}

\newtheorem{claim}{Claim}
\newtheorem{claim1}{Claim}
\newtheorem{question}[theorem]{Question}
\usepackage{dashrule}
\usepackage{tikz}
\usetikzlibrary{decorations.markings}
\newcommand{\dhorline}[3][0]{%
    \tikz[baseline]{\path[decoration={markings,
      mark=between positions 0 and 1 step 2*#3
      with {\node[fill, circle, minimum width=#3, inner sep=0pt, anchor=south west] {};}},postaction={decorate}]  (0,#1) -- ++(#2,0);}}
\begin{document}
\newcommand{\Addresses}{{
\bigskip
\footnotesize

\medskip

Maria-Romina~Ivan, \textsc{Department of Pure Mathematics and
Mathematical Statistics, Centre for Mathematical Sciences, Wilberforce
Road, Cambridge, CB3 0WB, UK.}\par\nopagebreak\textit{Email address:}
\texttt{mri25@dpmms.cam.ac.uk}

\medskip

Imre~Leader, \textsc{Department of Pure Mathematics and Mathematical
Statistics, Centre for Mathematical Sciences, Wilberforce Road,
Cambridge, CB3 0WB, UK.}\par\nopagebreak\textit{Email address:}
\texttt{i.leader@dpmms.cam.ac.uk}

\medskip

Luca Q.~Zamboni, \textsc{Universit\'e de Lyon, Universit\'e de Lyon 1,
CNRS UMR 5208, Institut Camille Jordan, 43 boulevard du 11 novembre
1918, F69622 Villeurbanne Cedex, France.}\par\nopagebreak\textit{Email
address:} \texttt{zamboni@math.univ-lyon1.fr}
}}
\pagestyle{fancy}
\fancyhf{}
\fancyhead [LE, RO] {\thepage}
\fancyhead [CE] {MARIA-ROMINA IVAN, IMRE LEADER, LUCA Q.
ZAMBONI}
\fancyhead [CO] {A RAMSEY CHARACTERISATION OF EVENTUALLY PERIODIC WORDS}
\renewcommand{\headrulewidth}{0pt}
\renewcommand{\l}{\rule{6em}{1pt}\ }
\title{\Large\textbf{A RAMSEY CHARACTERISATION OF EVENTUALLY PERIODIC
WORDS}}
\author{MARIA-ROMINA IVAN, IMRE LEADER, LUCA Q. ZAMBONI}
\date{}
\maketitle
\begin{abstract}
A factorisation $x = u_1 u_2 \cdots$ of an infinite word $x$ on alphabet
$X$ is called
`monochromatic', for a given colouring of the finite words $X^*$ on
alphabet $X$, if each $u_i$ is the
same colour.
Wojcik and Zamboni proved that the word $x$ is periodic if and only if
for every finite colouring
of $X^*$ there is a monochromatic factorisation of $x$. On the other
hand, it follows from
Ramsey's theorem that, for {\it any} word $x$, for every finite
colouring of $X^*$ there is a
suffix of $x$ having a monochromatic factorisation.\par A factorisation
$x = u_1 u_2 \cdots$ is called `super-monochromatic' if
each word
$u_{k_1} u_{k_2} \cdots u_{k_n}$, where $k_1 < \cdots < k_n$, is the
same colour. Our aim in this paper
is to show that a word $x$ is eventually periodic if and only if for
every finite
colouring of $X^*$ there is a suffix of $x$ having a super-monochromatic
factorisation.
Our main tool is a Ramsey result about
alternating sums that may be of independent interest.
\end{abstract}
\section{Introduction}Let $X$ be a non-empty finite or infinite set (called the alphabet) and let $X^*$ denote the set of all finite words $x_1x_2\cdots x_n$ with $n\geq 1$ and $x_i\in X$ for all $i$. Let $x=x_1x_2x_3\cdots$ be an infinite word on $X.$ Given a finite colouring of $X^*,$ we say that a factorisation $x=u_1u_2u_3\cdots $ (with $u_i\in X^*$) is \textit{monochromatic} if all the $u_i$ have the same colour. When is it
the case that $x$ always has a monochromatic factorisation, for any
finite colouring of $X^*$?\par This is certainly the case if $x$ is
\textit{periodic}. Indeed, if $x = uuu
\cdots$ then for any
colouring of $X^*$ that very factorisation is trivially monochromatic.
In the
other direction,
Wojcik and Zamboni \cite{Luca} proved that if $x$ is not periodic then
there
exists a finite
colouring of $X^*$ for which $x$ does not have a monochromatic
factorisation. Thus the above
Ramsey condition actually characterises the periodic words.\par We
remark that if we are allowed to pass to a suffix of $x$ then this
characterisation
breaks down completely. Indeed, {\it every} word $x$  has the property
that for every finite
colouring of $X^*$ there is a suffix of $x$ having a monochromatic
factorisation. This
result is due to Sch{\"u}tzenberger \cite{S}, and it follows from
Ramsey's
theorem. To see this, let
$x$ be the word $x_1 x_2 \cdots$. Given a
finite colouring $\phi$ of $X^*$, we define a colouring of
$\mathbb{N}^{(2)}$,
the edge set of the
complete graph on the natural numbers, by giving the pair $(i,j)$, where
$i<j$, the
colour $\phi(x_i x_{i+1} \cdots x_{j-1})$. By Ramsey's theorem, there is
a monochromatic infinite set for this colouring, say $m_1 < m_2 <
\cdots$. But now
we note that the finite words $x_{m_i} x_{m_i + 1} \cdots
x_{m_{i+1}-1}$, for each $i$,
are all assigned the same colour by $\phi$, and they form a
factorisation of the
suffix of $x$ starting at position $m_1$.

\par
Actually, the above argument shows that more: it shows that, for any
colouring, there is a suffix of $x$ having a factorisation $u_1 u_2
\cdots$
in which every word $u_i u_{i+1} \cdots u_j$, for $i \leq j$, has the
same colour. It was shown by de Luca and Zamboni \cite{LZ} that this
strengthened form is actually equivalent to Ramsey's theorem.

\par In light of these results,
it is natural to ask if there is a Ramsey
characterisation
of the \textit{eventually periodic} words over $X$, i.e., infinite words of the form $uvvv\cdots$ with $u,v\in X^*$.
We say that a factorisation $x = u_1 u_2 \cdots$ is {\it
super-monochromatic} if each word
$u_{k_1} u_{k_2} \cdots u_{k_n}$, where $k_1 < \cdots < k_n$, is the
same colour. Our motivation for
considering this notion comes from the following observation: if
$x$ is eventually periodic then for every finite colouring of $X^*$
there
is a suffix of
$x$ having a super-monochromatic factorisation.\par Indeed, given a
finite colouring $\phi$ of $X^*$, it suffices to take a
suffix of $x$ that
is periodic: say $y = uuu \cdots$. We induce a colouring of $\mathbb{N}$
by
giving the number $n$
the colour $\phi(u^n)$. By Hindman's theorem \cite{Fin}, there
exists an
infinite
set $M \subset \mathbb{N}$, say $M = \{a_1,a_2,\cdots \}$, where $a_1 <
a_2 <
\cdots$, such that
every (non-empty) finite sum of distinct elements of
$M$ has the same colour. But now the factorisation $y = u^{a_1} u^{a_2}
\cdots$ is
super-monochromatic.\par Our aim in this paper is to show that this
condition actually
characterises the
eventually periodic words. In other words, we will show that if the word
$x$ has the
property that for every finite
colouring of $X^*$ there is a suffix of $x$ having a super-monochromatic
factorisation, then
$x$ is eventually periodic.
\begin{theorem} Let $x$ be an infinite word on alphabet $X$. Then $x$ is
eventually periodic if
and only if for every finite
colouring of $X^*$ there is a suffix of $x$ having a super-monochromatic
factorisation.
\end{theorem}
(Note that if we `swap the quantifiers' we would have the statement that $x$ is
eventually periodic if and only if there is a suffix of $x$ such that every
finite colouring of this suffix has a super-monochromatic factorisation -- 
which is true by the remarks above.)

This result has actually been around in the community as a conjecture
for
some time (see e.g.~\cite{Wojphd}). There have been some partial
results, of which the strongest is
perhaps the result of Wojcik \cite{Wojphd}, who showed that Theorem 1 holds for words $x$ that have at most finitely many distinct square factors, where by a \textit{square factor} we mean a non-empty block of the form $uu$ which occurs in $x$. But the result was not even known for Sturmian words, which are
regarded as the `simplest' \textit{aperiodic} words, i.e., words that are not eventually periodic.
\par
Let us also remark that if one is allowed to pass to the shift orbit
closure
then the situation is completely different. Recall that the
{\it shift orbit closure} of an infinite word $x$ is the closure of the
set of suffices of $x$ in the product topology: equivalently, it
consists of all infinite words $y$ such that every factor of $y$ is a
factor of $x$. Van Th\'e and Zamboni (see \cite{Woj}) showed that,
for {\it any}
infinite word $x$ over a finite alphabet $X$, whenever $X^*$ is
finitely coloured there is a word $y$ in the shift orbit closure of $x$
having a super-monochromatic factorisation.

\par
Our proof is in two separate parts. In the first part, we reduce the
problem to a
problem that concerns not colourings of words, but colourings of
$\mathbb{N}^{(2)}$. It will
turn out from this reduction that Theorem 1 is implied by the following
result, which
concerns alternating sums and may be of independent interest.
\begin{theorem} There exists a finite colouring of $\mathbb{N}^{(2)}$
such that there
do {\it not}
exist $x_1 < x_2 < \cdots$ for which the set of all pairs
$(x_{k_1} - x_{k_2} + x_{k_3} - \cdots + x_{k_{t}}, x_{k_{t+1}})$, where
$t$ is odd and
$k_1 < k_2 < \cdots < k_{t+1}$, is monochromatic.
\end{theorem}

The second part of the proof thus consists of a proof of Theorem 2. What
is interesting
is the role played by the alternation. Indeed, if all the signs were
plus-signs then
the Ramsey statement would be in the affirmative. In other words, for
any
finite colouring of $\mathbb{N}^{(2)}$ there exist $x_1 < x_2 < \cdots$
for
which the set of
all pairs $(x_{k_1}+x_{k_2}+x_{k_3}+\cdots+ x_{k_{t}},
x_{k_{t+1}})$,
where $k_1 < k_2 < \cdots < k_{t+1}$, is monochromatic. This follows for
example from the
Milliken-Taylor theorem (\cite{Mik}, \cite{Taylor}). To see this, recall
that the Milliken-Taylor theorem asserts that whenever the set of all
pairs $(A,B)$, where $A$ and $B$ are (non-empty) finite subsets of
$\mathbb{N}$ with $\max A < \min B$, is finitely coloured there
exists a sequence $A_1,A_2, \cdots$ of finite subsets of
$\mathbb{N}$, with $\max A_n < \min A_{n+1}$ for all $n$, such that all
of the pairs $(S,T)$, where $S$ and $T$ are finite unions of the $A_n$
with
$\max S < \min T$, are the same colour. So we just need to `transfer'
the
colouring from numbers to finite sets: given a finite colouring $\theta$
of
$\mathbb{N}^{(2)}$, we colour each pair $(A,B)$ as above with the colour
$\theta(\sum_{i \in A} 2^i, \sum_{i \in B} 2^i)$. Given the sequence
$A_1,A_2, \cdots$ as guaranteed by the Milliken-Taylor theorem, we set
$x_n = \sum_{i \in A_n} 2^i$, and now we get that every pair
$(x_{k_1}+x_{k_2}+x_{k_3}+\cdots+ x_{k_{t}},
x_{k_{t+1}})$, and in fact even every pair
$(x_{k_1}+x_{k_2}+x_{k_3}+\cdots+ x_{k_{t}},
x_{k_{t+1}}+x_{k_{t+2}} + \cdots +x_{k_s})$, has the same colour.
The interested
reader is referred
to \cite{HS} for
a general discussion of the Milliken-Taylor theorem and many related
results, although we stress that this paper is self-contained.

\par
The colouring argument needed to establish Theorem 2 is rather
complicated, and it is perhaps worthwhile to
describe why this is the case. As we will see, it turns out to be
useful to `change variables' to some other variables, the $y_n$, that
satisfy a related condition. However, this related condition is not
preserved by
passing to subsequences. This is in contrast to the usual situations
when
one is finding a `bad' colouring (see for example \cite{Deuber},
\cite{Imre}, \cite{LRW}), where the first step is always to pass to a
subsequence
or sequence of sums in which the supports of the elements, when written
say in
binary, are disjoint, and even more are ordered in the sense that one
variable's support ends before the next one's begins. Since this step is
not available to us here, we have to deal with the situation when the
$y_i$ do not have disjoint supports, and therefore we need to consider
how the carry-digits behave when we add them to each other. This means
that
the colouring, and especially the proof that it works, is far more
difficult than for other problems that superficially look similar.

\par The plan of the paper is as follows. In Section 2 we show
that Theorem 1 is implied
by Theorem 2, and then in Section 3 we prove Theorem 2. Section 4 is
devoted to some
related problems that we have been unable to solve.
\section{The link between the two theorems}
\par In this section we will show that Theorem 2 implies Theorem 1. As
explained above, we know that given any finite colouring of $X^*$, any
eventually periodic word has a suffix that admits a super-monochromatic
factorisation. Therefore, we only need to prove the reverse implication
of Theorem 1: given an aperiodic word $x$ (i.e.~$x$ is not eventually periodic), we must
construct a finite colouring of $X^*$ for which no suffix of $x$ has a
super-monochromatic factorisation. This will be accomplished using the
colouring given by Theorem 2.
\begin{theorem}
Theorem 2 implies Theorem 1.
\end{theorem}
\begin{proof}
By Theorem 2, there exists a finite colouring of $\mathbb N^{(2)}$,
$\theta$, for which there is no increasing sequence
$(x_k)_{k\geq1}$ such that all edges of the form
$(x_{k_1}-x_{k_2}+\cdots +x_{k_t}, x_{k_{t+1}})$ have the same colour,
where $k_1<k_2<\cdots <k_{t+1}$ and $t$ is odd. Let $\mathcal C$ be the
set of colours of $\theta$.\\
\\
Let $x$ be an aperiodic word. We will use $\theta$
to construct a finite colouring $\phi$ of $X^*$ for which no suffix of
$x$ has a super-monochromatic factorization. We denote by $x_i$ the
$i^{\text{th}}$ letter of $x$.\\
\\
For any factor $u$ of $x$, define $A_x(u)=\text{min}\{n\in\mathbb N:
u=x_nx_{n+1}\cdots  x_{n+|u|-1}\}$ and $B_x(u)=A_x(u)+|u|$, where $|u|$
is the length of $|u|$. In other words, $A_x(u)$ is the start position
of the first occurrence of $u$ in $x$, while $B_x(u)$ is the first
position after this first occurrence of $u$.\\For an arbitrary
factorisation $(u_i)_{i\geq1}$ of $x$, we say $(w_i)_{i\geq1}$ is a
\textit{block subfactorisation} of $(u_i)_{i\geq 1}$ if there exists a
strictly increasing sequence of positive integers $(k_j)_{j\geq 1}$ such that
$w_1=u_1u_2\cdots u_{k_1}$ and $w_i=u_{k_{i-1}+1}\cdots u_{k_i}$ for
each $i\geq 2$. Here by $(u_i)_{i\geq1}$ being a factorisation of $x$ we
mean $x=u_1u_2u_3\cdots$. We immediately note that a block
subfactorisation of a super-monochromatic factorisation is still
super-monochromatic.
\\
\\
Now we are ready to define a colouring $\phi:X^*\rightarrow (\mathcal
C\times\{0,1\}) \cup\{2\}$ as follows:
\begin{enumerate}
\item If $u$ is not a factor of $x$, then $\phi(u)=2$.
\item If $u$ is a factor of $x$ and there exists a factorisation $u=vw$
such that $A_x(u)=A_x(v)$ and $B_x(u)=B_x(w)$ (in other words, the first
occurrence of $v$ in $x$ is as the start of the first occurrence of $u$ in
$x$ and also the first occurrence of $w$ in $x$ is as the end of the
first occurrence of $u$ in $x$), then $\phi(u)=(\theta(A_x(u),
B_x(u)),0)$.
\item Otherwise $\phi(u)=(\theta(A_x(u), B_x(u)), 1)$.
\end{enumerate}
We claim that for this colouring $\phi$ no suffix of $x$ has a
super-monochromatic factorisation.\\
Suppose to the contrary that there is a suffix $y$ of $x$ having a super-monochromatic factorisation $y=u_1u_2\cdots$. Let $u_0$ be the (possibly empty) prefix of $x$ so that $x=u_0y$. It is important to remember that each factor $u_i$ may occur in several
places in $y$, not necessarily only in the place immediately following
$u_1u_2\cdots u_{i-1}$. We call this place the \textit{standard}
position of $u_i$. Let the colour of all concatenations of the $u_i$ be
$c\in ( \mathcal C\times\{0,1\}) \cup\{2\}$. Since $\phi(u_1)=c$, and since
$u_1$ is a factor of $x$, we have $c\neq 2$. Thus, $c=(a,b)$ where
$a\in\mathcal C$ and $b\in\{0,1\}$.
\begin{claim} By passing to a block subfactorisation, we may assume that
for every $i\in\mathbb N$, the first occurrence of $u_i$ in $x$ is
exactly the standard position of $u_i$.\end{claim} Equivalently, this
means $A_x(u_i)=|u_0|+|u_1|+\cdots +|u_{i-1}|+1$ for all $i\geq 1$.
\begin{proof} We start by showing that we may assume $A_x(u_1)=|u_0|+1$.
If initially $A_x(u_1)<|u_0|+1$, we consider all concatenations
$u_1u_2\cdots u_k$. If $A_x(u_1u_2\cdots u_k)=|u_0|+1$ for some $k$, we
set our first factor to be $u_1u_2\cdots u_k$ and renumber the rest of
them. Since concatenating consecutive factors does not change the
super-monochromatic property, the new factorisation is still
super-monochromatic and the first factor now has the desired
property.\\If on the other hand $A_x(u_1u_2\cdots u_k)<|u_0|+1$ for all $k\geq1$, then each concatenation $u_1u_2\cdots u_k$ first occurs in $x$ starting at some position in $u_0$. Since there are
infinitely many of them and only finitely many positions in $u_0$, there
exists a position $i$, with $i\leq|u_0|$, at which infinitely many
$u_1u_2\cdots u_k$ start. This immediately implies that the suffix of
$x$ starting at position $i$ is exactly $y$. But this means that $x$ has
two suffices equal to $y$, which implies that $x$ is eventually
periodic. More precisely, we have $x_1x_2\cdots x_{i-1}y=x_1x_2\cdots
x_{|u_0|}y$. Therefore $y_k=x_{i-1+k}$ and $y_k=x_{|u_0|+k}$ for any
$k$. It follows that $x_{i-1+k}=x_{|u_0|+k}$ for any $k$, thus $x$ is
eventually periodic with period $|u_0|-i+1$, contradicting our initial
assumption.\\Therefore we may assume $u_1$ has the desired property. We
now move on to $u_2$ and repeat the same argument, looking at
concatenations of the form $u_2u_3\cdots u_k$: so $u_2$ may be assumed
to have the same property too. It follows inductively that we may assume
that all $u_i$ have the property stated in the claim.
\end{proof}
We further observe that once we have the property that the first
occurrence of each $u_i$ in $x$ is in the standard position, then any
block subfactorisation has this property as well. For example, $u_1u_2$
cannot appear earlier or else $u_1$ would. Therefore, we can further
assume that $|u_{n+1}|\geq|u_1u_2\cdots u_n|$ for all $n\geq1$.
\par We now look at $u_1u_2$. Because of the above claim we certainly
have $A_x(u_1u_2)=A_x(u_1)$ and $B_x(u_1u_2)=B_x(u_2)$. This means that
$u_1u_2$ is a factor of $x$ that satisfies the factorisation condition
specified in the colouring rule. Thus
$\phi(u_1u_2)=(a,b)=(\theta(A_x(u_1u_2),B_x(u_1u_2)),0)$, and so the
colour of the factorisation is $(a,0)$ with $a\in\mathcal C$.
\begin{claim} The word $u_1u_2\cdots u_n$ is a suffix of $u_{n+1}$, for
every $n\geq 1$.
\end{claim}
\begin{proof}
Consider the concatenation $u_1u_2\cdots u_nu_{n+2}$. Because our
factorisation $u_1u_2\cdots$ is super-monochromatic we have that
$\phi(u_1u_2\cdots u_nu_{n+2})=(a,0)$. This means that not only is
$u_1u_2\cdots u_n u_{n+2}$ a factor of $x$, but also that $u_1u_2\cdots
u_nu_{n+2}=vw$ for some $v$, $w$ with $A_x(u_1u_2\cdots
u_nu_{n+2})=A_x(v)$ and $B_x(u_1u_2\cdots u_nu_{n+2})=B_x(w)$.\\We now
have two possibilities: either $v$ is a prefix of $u_1u_2\cdots u_n$ or
$u_1u_2\cdots u_n$ is a prefix of $v$.\\If $v$ is a prefix of
$u_1u_2\cdots u_n$ then $u_{n+2}$ is a suffix of $w$. Therefore we
immediately have that $A_x(v)\leq A_x(u_1u_2\cdots u_n)$ and $B_x(w)\geq
B_x(u_{n+2})$. It follows that \begin{center}
$B_x(u_1u_2\cdots u_nu_{n+2})=B_x(w)\geq B_x(u_{n+2})=B_x(u_1u_2\cdots
u_nu_{n+1}u_{n+2})$\end{center}and\begin{center}$A_x(u_1u_2\cdots
u_nu_{n+2})=A_x(v)\leq A_x(u_1u_2\cdots u_n)=A_x(u_1u_2\cdots
u_nu_{n+1}u_{n+2})$,
\end{center}
where the last equalities in each line follow from the property that,
for each $i$, the first occurrence of $u_i$ in $x$ is at its standard
position. Consequently, any consecutive concatenation of such factors
has the same property.\\
Putting these two inequalities together, we obtain\begin{center}
$B_x(u_1u_2\cdots u_nu_{n+2})-A_x(u_1u_2\cdots u_nu_{n+2})\geq
B_x(u_1u_2\cdots u_nu_{n+1}u_{n+2})-A_x(u_1u_2\cdots
u_nu_{n+1}u_{n+2})$.\end{center} This is equivalent to $|u_1u_2\cdots
u_nu_{n+2
}|\geq|u_1u_2\cdots u_nu_{n+1}u_{n+2}|$, which is a contradiction.
\\Hence $u_1u_2\cdots u_n$ is a prefix of $v$, and so $w$ is a suffix of
$u_{n+2}$. This implies that $B_x(w)\leq B_x(u_{n+2})$. Since $u_{n+2}$
is a suffix of $u_1u_2\cdots u_nu_{n+2}$, the same argument gives
$$B_x(u_{n+2})\leq B_x(u_1u_2\cdots u_nu_{n+2})=B_x(w).$$ Therefore
$B_x(u_{n+2})=B_x(u_1u_2\cdots u_nu_{n+2})$. By Claim 1 we also have
$B_x(u_{n+1}u_{n+2})=B_x(u_{n+2})$. We conclude that
$B_x(u_{n+1}u_{n+2})=B_x(u_1u_2\cdots u_nu_{n+2})$ which, combined with
$|u_{n+1}|\geq|u_1u_2\cdots u_n|$, gives that $u_1u_2\cdots u_n$ is a
suffix of $u_{n+1}$.
\end{proof}
\begin{claim} The word $u_{k_1}u_{k_2}\cdots u_{k_m}$ is a suffix of
$u_n$, for every $k_1<k_2<\cdots <k_m<n$.
\end{claim}
\begin{proof} We prove the statement by induction on the number of
factors.\\
From Claim 2 we get that $u_t$ is a suffix of $u_{t+1}$ for every
$t\geq1$. Since `is a suffix of' is a transitive property, we obtain
that $u_t$ is a suffix of $u_n$ for every $t<n$, thus the base case is
proved. \\Assume now that the result is true for all concatenations of
at most $s$ factors, and consider a concatenation $u_{k_1}u_{k_2}\cdots
u_{k_s}u_{k_{s+1}}$ with all $k_i<n$. If the indices are consecutive
numbers, Claim 2 guarantees that this is a suffix of $u_{k_{s+1}+1}$,
which is a suffix of $u_n$. If that is not the case, then
$k_i+1<k_{i+1}$ for some $i\leq s$. We take $i$ to be the biggest such
index and apply the induction hypothesis to obtain that $u_1u_2\cdots
u_{k_i}$ is a suffix of $u_{k_{i}+1}$, which is a suffix of
$u_{k_{i+1}-1}$. It then follows that $u_{k_1}u_{k_2}\cdots
u_{k_s}u_{k_{s+1}}$ is a suffix of the consecutive concatenation of
factors $u_{k_{i+1}-1}u_{k_{i+1}}\cdots u_{k_{s+1}}$, which is a suffix
of $u_{k_{s+1}+1}$, thus a suffix of $u_n$. This finishes the inductive
step and hence proves the claim.
\end{proof}

Combining Claim 1 and Claim 3, we obtain that $B_x(u_{k_1}u_{k_2}\cdots
u_{k_t})=B_x(u_{k_t})$. This is because repeatedly applying Claim 3 tells us that
$u_{k_1}u_{k_2}\cdots u_{k_t}$ is a suffix of a consecutive
concatenation of factors ending in $u_{k_t}$. Note that, by
construction, we also have $A_x(u_{n+1})=B_x(u_n)$.
\par We now return to our original colouring. By assumption, we have
that \begin{center}$\theta(A_x(u_{k_1}u_{k_2}\cdots u_{k_t}),
B_x(u_{k_1}u_{k_2}\cdots u_{k_t}))=a$\end{center} for every
$k_1<k_2<\cdots <k_t$. We also know that
\begin{align*}A_x(u_{k_1}u_{k_2}\cdots u_{k_t})&
=B_x(u_{k_1}u_{k_2}\cdots u_{k_t})-|u_{k_1}u_{k_2}\cdots u_{k_t}|\\
&=B_x(u_{k_t})-|u_{k_1}|-|u_{k_2}|-\cdots -|u_{k_t}|\\
&=A_x(u_{k_t})+A_x(u_{k_{t-1}})+\cdots
+A_x(u_{k_1})-B_x(u_{k_{t-1}})-\cdots -B_x(u_{k_1}),\end{align*}
where we used the fact that $|u_{k_i}|=B_x(u_{k_i})-A_x(u_{k_i})$.\\Let
$m_i=B_x(u_i)$ for each $i$. Clearly $(m_i)_{i\geq1}$ is a strictly
increasing sequence. We then have\begin{center}$A_x(u_{k_1}u_{k_2}\cdots
u_{k_t})=m_{k_t-1}+m_{k_{t-1}-1}+\cdots +m_{k_1-1}-m_{k_{t-1}}-\cdots
-m_{k_1}.$\end{center}
It follows that for any choice of $k_1<k_2<\cdots <k_t$, we have
that\begin{center}$\theta(m_{{k_1}-1}-m_{k_1}+m_{k_2-1}-m_{k_2}+\cdots
+m_{k_{t-1}-1}-m_{k_{t-1}}+m_{k_t-1}, m_{k_t})=a$.\end{center}
By choosing the $k_i$ appropriately, it follows that that for any $l$
odd and $i_1<i_2<\cdots <i_l<i_{l+1}$, we have
$\theta(m_{i_1}-m_{i_2}+\cdots -m_{i_{l-1}}+m_{i_l}, m_{i_{l+1}})=a$,
which contradicts the choice of $\theta$.
\end{proof}
\section{Constructing the colouring $\theta$}
In this section we will construct a finite colouring of $\mathbb N
^{(2)}$ with the property that for no infinite strictly increasing
sequence $(x_n)_{n\geq 1}$ do all pairs of the form
$(x_{k_1}-x_{k_2}+\cdots-x_{k_{t-1}}+x_{k_t}, x_{k_{t+1}})$ have the
same colour, where $k_1<k_2<\cdots<k_{t+1}$ and $t$ is odd.
\par We start with a simple observation. Let $y_1=x_1$ and
$y_n=x_n-x_{n-1}$ for each $n\geq 2$. So $x_n=y_n+y_{n-1}+\cdots+y_1$.
\par Now let $t$ be odd and $k_1<k_2<\cdots<k_t$. We then have that $x_{k_1}-x_{k_2}+\cdots-x_{k_{t-1}}+x_{k_t}=x_{k_1}+(x_{k_3}-x_{k_2})+\cdots+(x_{k_t}-x_{k_{t-1}})$. Thus $x_{k_1}-x_{k_2}+\cdots-x_{k_{t-1}}+x_{k_t}=y_1+y_2+\cdots+y_{k_1}+(y_{k_2+1}+\cdots+y_{k_3})+\cdots+(y_{k_{t-1}+1}+\cdots+y_{k_t}).$
\par Let $1<m_1<\cdots<m_s$ be integers and set in the above expression $k_1=1$, $k_2+1=k_3=m_1,$ $\cdots$, $k_{2s}+1=k_{2s+1}=m_s$. Then we obtain that $x_{k_1}-x_{k_2}+\cdots-x_{k_{t-1}}+x_{k_t}=y_1+y_{m_1}+\cdots+y_{m_s}$. This shows that Theorem 2 is equivalent to:
\begin{theorem}
There exists a finite colouring of $\mathbb N^{(2)}$ such that there
does not exist a sequence of natural numbers $(y_k)_{k\geq 1}$ for
which all pairs of the form $(y_{1}+y_{k_1}+y_{k_2}+\cdots  +y_{k_t},
y_1+y_2+y_3+\cdots+y_{k_{t+1}})$ have the same colour, for all choices
of
$1<k_1<k_2<\cdots <k_{t+1}$.
\end{theorem}
\begin{proof}
Our construction of the colouring will be in several stages. At each
stage, we add more colours, meaning that we take the product colouring
of the colouring we have so far with a new colouring. The conditions on
a supposed sequence $(y_n)_{n\geq 1}$ satisfying the conditions in
Theorem 4 will thus become more and more stringent, eventually resulting
in a
contradiction.
\par

As the colouring is rather complex, we give a brief
overview of what each stage is supposed to achieve. We first need some
notation. We work with natural numbers in their binary form, so
strings of `0' and `1'. The \textit{position} of a digit is the power of
2 it represents. The \textit{first} digit of $n$ in binary is at
position $i$, where $i$ is the greatest non-negative integer such that
$2^{i}$ divides $n$. The \textit{last} digit of $n$ in binary is at
position $j$, where $2^j\leq n<2^{j+1}$. The \textit{support} of $n$ is
the set of positions having the digit `1' in its binary expansion. For
example, let
$n=2^7+2^6+2^3$. Below $n$ is shown in binary, where the first row
represents the position of each digit. The support of $n$ is
$\{3,6,7\}$.
\begin{center}

\begin{tabular}{|c|c|c|c|c|c|c|c|c|}
\hline
Position number&7&6&5&4&3&2&1&0\\
\hline
Binary digit of $n$&1&1&0&0&1&0&0&0\\
\hline
&last digit& & & &first digit& & & \\
\hline

\hline
\end{tabular}
\end{center}
\par
In Stage 1 we will
ensure that the supports (in binary) of the $y_n$ do roughly `go off to
the left'. What we hope to achieve is a `staircase' pattern. More
precisely,
writing $n_i$ for the position of the first digit of $y_i$ and $m_i$ for the position of its last digit, we would like to ensure that the $n_i$ and the $m_i$ form
strictly increasing sequences, with $y_{i+2}$ starting to the left of
where $y_i$ ends for all $i$. This is the idea behind the definition of
`Type A' below. However, it will turn out that we cannot always achieve
this, and so there is a residual case to deal with that we call `Type
B',
which represents what happens when there is no way to pass to Type A. Of
course, we cannot ignore this case, but somehow it has the feel of an
annoying special case: the reader should perhaps view Type A as the
`main'
case. Roughly speaking, the sequence is of Type B when the sequence,
with
$y_1$ removed, is of Type A, but also $y_1$ starts where $y_2$ starts,
and
$y_1+y_2$ starts where $y_3$ starts, and $y_1+y_2+y_3$ starts where
$y_4$ starts, and so on.
\par
Then Stage 2 gives that the supports of the $y_i$, despite having the
above
staircase pattern, cannot be disjoint. And it
then starts to deal with the unpleasant issues arising from `carry
digits', that arise when adding numbers whose supports are not disjoint.
It will
give that the carries must be short-range, in a certain sense. The fact
that we forbid the carries to propagate arbitrarily far will actually
show that Type B cannot occur. And
finally, Stage 3 will eliminate these short-range carries as well.
\par

We are now ready to turn to the proof itself. As stated above, we
construct our colouring step by step. When colouring
a pair $(a,b)$, $a<b$, we will often look at just $a$, just $b$, or just
$b-a$. At other times we will make full use of the fact that we are
colouring pairs, not just numbers.  \\
\par \textit{Stage 1.} To start with, our colours are quadruples $(c_0,
c_1, c_2, c_3)$ where $c_0, c_1, c_2\in\{0,1,2\}$ and $c_3$ is one of
the four possible bit-strings of length 3 having `1' at their rightmost
positions. We give the colour
$(c_0,c_1, c_2, c_3)$ to the pair $(a,b)$, $a<b$, if the last digit of
$b-a$ in binary is at position $c_0$ modulo 3, the first digit of $b-a$
in binary is at position $c_1$ modulo 3, the last digit of $a$ in binary
is at position $c_2$ modulo 3, and the first 3 digits of $b-a$ form,
from left to right, $c_3$.\\
Assume now $(y_n)_{n\geq 1}$ is a sequence that satisfies the conditions
of Theorem 4 for the above colouring. The pairs that are all the same
colour are the pairs of the form
$(y_1+y_{k_1}+\cdots+y_{k_{t-1}}+y_{k_t},
y_1+y_2+\cdots+y_{k_{t+1}})$ for some $t\geq 0$ and $1<k_1<\cdots<k_t<k_{t+1}$. The differences $b-a$, where $(a,b)$ is a pair of the above form, are precisely the sums of the form $y_{k_1}+\cdots +y_{k_t}$ for some $t\geq1$ and $1<k_1<\cdots<k_t$. It follows that any such
finite sum must have the same first 3 digits, with the first digit at a
fixed position modulo 3.
\begin{claim1}
There do not exist $j>i>1$ such that $y_j$ and $y_i$ have their first
digit at the same position.
\end{claim1}
\begin {proof}
Assume that two such $y_i$ and $y_j$ exist. The position of their
first digit is the same, say $n_0$, and by our colouring they have the
same first 3 digits. On the other hand, the colouring also requires
$y_i+y_j$
to have the first digit at a position congruent to $n_0$ modulo 3.
However, adding two identical strings in binary shifts the support by
exactly one to the left. Hence, when we add $y_i$ and $y_j$, their first
`1', which was at position $n_0$ for both of them, is moved to position
$n_0+1\not\equiv n_0$ modulo 3, a contradiction.
\end{proof}
We now know that, except for possibly $y_1$, no two terms of the sequence start
at the same position.\\ \par Let $(z_n)_{n\geq 1}$ be a sequence of
natural numbers. We call $(w_n)_{n\geq 1}$ a \textit{full block
subsequence} or simply a  \textit{block subsequence} of $(z_n)_{n\geq
1}$
if there exists an increasing sequence of natural numbers $(k_n)_{n\geq 1}$ such
that $w_1=z_1+\cdots+z_{k_1}$ and $w_n=z_{k_{n-1}+1}+\cdots+z_{k_n}$ for
$n\geq 2$. We stress that there are no `gaps': every $z_n$ appears as a
summand in some $w_m$.\\We observe that if the sequence $(y_n)_{n\geq
1}$ satisfies the conditions in Theorem 4 for a given colouring then so
does any of its block subsequences. So, by passing to a block
subsequence, we may assume that $(y_n)_{n\geq 1}$ is strictly
increasing.\par Let $(z_n)_{n\geq 1}$ be a sequence of natural numbers.
Let $n_i$ be the position of the first digit
of $z_i$ and $m_i$ the position of the last digit of $z_i$, for all
$i\geq 1$. We call the sequence $(z_n)_{n\geq1}$ of
\textit{Type A} if for all $i\geq1$, $n_i<n_{i+1}$, and $m_i<m_{i+1}$,
and
$m_i+1<n_{i+2}$. We call the sequence $(z_n)_{n\geq1}$
of \textit{Type B} if none of its block subsequences is of Type A, the
sequence $(z_n)_{n\geq 2}$ is of Type A, and also $n_1=n_2$, and
$m_1<m_2$, and
$m_1+1<n_3$. We remark that this definition of  `Type B' is more
abstract that
the one informally described in the proof overview above: the reason is
that
we want this definition to capture the idea of `we cannot pass to Type
A'.
\begin{claim1}
By passing to a block subsequence, we may assume that $(y_n)_{n\geq1}$
is of either Type A or Type B.
\end{claim1} \begin{proof} As above, let the positions of the first and
last digits of $y_i$ be $n_i$ and $m_i$ respectively.\par Assume first
that there is no $k$ such that the first digit of $y_k$ is at the same
position as the first digit of $y_1$. We will prove that we can find a
block subsequence of Type A. We start with $y_1$. By Claim 1, only
finitely many terms have the position of their first digit at most the
position of the first digit of $y_1$. Let $y_{l_1}$ be the last one of
them. We replace $y_1$ by the consecutive sum $y_1+y_2+\cdots+y_{l_1}$
and relabel the sequence accordingly. Now we move on to the second term.
All terms after $y_1$ now have the position of their first digit greater
than that of $y_1$. Again, only finitely many terms have their first
digit at a position at most one plus the position of the last digit of
$y_1$. Let $y_{l_2}$ be the last one of them. We now replace $y_2$ by
$y_2+y_3+\cdots+y_{l_2}+y_{l_2+1}$ and again relabel the sequence. Now
all terms after $y_2$ have their first digit at a position greater than
one plus the position of the last digit of $y_1$. Also, only finitely
many have the position of their first digit at most one plus the position of the
last digit of $y_2$. Let $y_{l_3}$ be the last one of them. We replace
$y_3$ by $y_3+y_4+\cdots+y_{l_3+1}$. Now continue inductively. Hence, we
obtain a block subsequence of Type A.\par We now assume that
$(y_n)_{n\geq1}$ does not have any block subsequence of Type A.
Therefore, there is a $k$ such that the first digit of $y_1$ is at the
same position as the first digit of $y_k$. We now construct a block
subsequence of Type B.\\First we note that we may assume that there is no $i>1$ such that
$n_i < n_1$: if such an $i$ exists, then we replace $y_1$ with
$y_1+y_2+\cdots+y_i$ and relabel. This new block subsequence has the
property that the first digits of its terms are all on different
positions. Therefore, by the argument presented at the begining of the
proof, we can construct a block subsequence of Type A, which is a
contradiction.\\We fix $y_1$. We know that no $y_i$ starts before it and
only $y_k$ starts at the same position. Only finitely many $y_i$ have
their first digit at a position at most one plus the position of the
last digit of $y_1$. Let $y_s$ be the last of them and let
$t=\text{max}(s+1,k)$. We now replace $y_2$ with $y_2+\cdots+y_t$.
Note that in this block subsequence $y_1$ and $y_2$ start on the same
position, and $m_1<m_2$. Since from now on the terms start on different
positions, we repeat the inductive construction presented at the beginning
of the proof and thus obtain the desired block subsequence of Type B.
\end{proof}
\par In what follows, a property that will play a crucial role is the
fact that any block subsequence of $(y_n)_{n\geq1}$ still satisfies
Claim 2. More precisely, as we now show, Type A sequences are invariant
under taking
block subsequences, and the same holds for Type B sequences. That is a
direct consequence of binary addition and our colouring so
far.
\begin{claim1} If $(y_n)_{n\geq1}$ satisfies
the conditions in Theorem 4 for the above colouring and is of Type A,  then the
same also holds for each of its block subsequences, and similarly for Type B. \end{claim1}
\begin{proof} Consider a sum $y_m+y_{m+1}+\cdots+y_n$, where $2\leq
m\leq n$. In any given position, at most two of the summands have a
digit 1 and so the last digit of the sum is either at the same position
as the last digit of $y_n$, or one greater. Because of the colour $c_0$,
we conclude that the last digit of $y_m+y_{m+1}+\cdots+y_n$ is at the
same position as the last digit of $y_n$.\\If $n\geq 3$, the sum
$y_1+y_2+\cdots+y_n$ has the last digit at the same position modulo 3 as
$y_1+y_n$, by $c_2$. Since $y_n$ and $y_1$ have disjoint supports, the
last digit of $y_1+y_n$ is at the same position as the last digit of
$y_n$. Similarly as above, the last digit of the sum
$y_1+y_2+\cdots+y_n$ is either at the same position as the last digit of
$y_n$, or one position greater. We conclude that the last digit of
$y_1+y_2+\cdots+y_n$ is at the same position as the last digit of $y_n$,
for $n\geq 3$. \\Finally, we look at $y_1+y_2$. Its last digit is either
at the same position as the last digit of $y_2$, or one position
greater. By $c_2$, the position of its last digit has to agree modulo 3
with the position of the last digit of $y_1+y_3$, which is the position
of the last digit of $y_3$, by disjointness. However, by $c_0$, $y_2$
and $y_3$ have the last digit at the same position modulo 3. We conclude
that the last digit of $y_1+y_2$ is at the same position as the last
digit of $y_2$.\\Thus, we have that for any $1\leq m\leq n$, the
position of the last digit of $y_m+\cdots+y_n$ is the position of the
last digit of $y_n$.\\ If the sequence $(y_n)_{n\geq1}$ is of Type A,
then the position of the first digit of $y_m+\cdots+y_n$ is the position of the first digit of $y_m$ for all $1\leq m\leq n$. Combining these two observations
we obtain that by passing to a block subsequence, we also obtain a Type
A sequence.\\If the sequence is of Type B, the first digit of
$y_m+\cdots+y_n$ is at the same position as the first digit of $y_m$ if
$m>1$. If $m=1$, then $y_1+\cdots+y_n$ has to start at the same position
as some other term $y_t + y_{t+1} + \cdots + y_{t+s}$, where $t \geq n+1$. 
This is because $(y_n)_{n\geq1}$ is of Type B and
thus cannot have any block subsequence of Type A, which can always be
constructed from a sequence with terms starting at different positions.
Since the last digit of this sum is at the same position as the last digit of $y_n$ which is
less than the position of the first digit of $y_{n+2}$, we must have that the first
digit of $y_1+\cdots+y_n$ is at the same position as the first digit of $y_{n+1}$. This shows
that any block subsequence is of Type B.\end{proof} We note that Claim 3 
also implies that the last digit of any sum is at the same position as the last digit of its
biggest term.\\
\par\textit{Stage 2.} Let $a$, $b$ with $a<b$ be a pair of natural
numbers. We write $a$ and $b$ in binary and we call a position $i$ a `2'
if both $a$ and $b$ have at position $i$ the digit 1. We call a position
$i$ a `1' if exactly one of $a$ and $b$ has at position $i$ the digit 1.
We define the \textit{number of `2 to 1'-jumps} of $(a,b)$, denoted by
$J(a,b)$, to be the number of transitions from a `2' to a `1' as we
traverse the positions in increasing order, ignoring the positions where
both numbers have a `0'. For example, if $a=100000100$ and $b=1101010111$,
then the positions labelled `1' are 0,1,4,6 and 9, the positions labelled `2'
are 2 and 8 and the positions ignored are 3,5 and 7. Thus the number of `2
to 1'-jumps is 2, namely the jump from position 2 to position 4 and the
jump from position 8 to position 9.\par Let $c$ be a natural number and
$c=l_pl_{p-1}\cdots l_1$ its binary representation. We call a binary
string
not containing a `0', $a_s\cdots a_1=1\cdots1$, an \textit{interval} of
$c$ if there exits $1\leq i\leq p-s+1$ such that $l_{i+s-1}\cdots
l_{i}=a_s\cdots
a_1$, $l_{i-1}=0$ or $i=1$, and $l_{i+s}=0$ or $i+s=p+1$. We denote by
$I(c)$ the number of intervals of $c$, counted with multiplicity. For
example, if $c=11101110010101$, then $I(c)=5$ since $c$ has two
intervals of length 3 and three intervals of length 1.\par We
now incorporate this into the colouring: we define a new colouring by
colouring $(a,b)$ by $(c_0,c_1,c_2,c_3,c_4,c_5)$ where $c_0$, $c_1$,
$c_2$ and $c_3$ are defined above, and $c_4=J(a,b)\mod 2$,
$c_5=I(b-a)\mod2$, with $c_{4},c_5=0$ or 1. So if $(y_n)_{n\geq 1}$
satisfies the conditions in Theorem 4 for this new colouring then it has
all the properties we have already established, in addition to any new
properties that may be forced by the new part of the colouring.\\We say
that two numbers $a$ and $b$, with $a<b$, have {\it right to left
disjoint supports} if the last digit of $a$ is at a position smaller
than the position of the first digit of $b$.
\begin{claim1} By passing to a block subsequence, we may assume  that
there is no $i\in\mathbb N$ such that both the pair $y_i,y_{i+1}$ and
the pair $y_{i+1},y_{i+2}$ have right to left disjoint supports.
\end{claim1}
\begin{proof} Assume that such an $i$ exists. As the cases $i=1$ and $i=2$ are slightly
different to the general case, we analyse them separately.
\begin{enumerate}
\item If $i=1$, we look at the colour of the pairs $(y_1+y_3,
y_1+y_2+y_3+y_4)$ and $(y_1+y_2+y_3, y_1+y_2+y_3+y_4)$, which have to be
the same colour. In particular, the value of $ J\mod 2$ has to be the
same. However, when we add $y_2$ to $y_1+y_3$, we eliminate exactly one
`2 to 1'-jump, namely the one where we moved from the last `2' in the
support of $y_1$ to the first `1' in the support of $y_2$. By the
disjointness of the supports, $y_2$ does not interact with $y_1$ or
$y_3$, so indeed the value of $J$ changes by exactly 1, a contradiction.
\item If $i=2$, we do not necessarily have that the supports of $y_1$
and $y_2$ are right to left disjoint. However, we observe that our
colouring requires that the position of the last digit of $y_1+y_2+y_3$
is the position of the last digit of $y_3$. This tells us that
$y_1+y_2+y_3$ and $y_4$ have right to left disjoint supports. By
replacing $y_1$ with $y_1+y_2+y_3$ and relabelling the rest of the
sequence, we may assume that $y_1$ and $y_2$ have right to left disjoint
supports. With this assumption, if $y_2$, $y_3$ and $y_4$ still have
right to left disjoint supports, we see that (with this choice of 
block subsequence) we are back in Case 1.
\item If $i>2$, then we look at the colour of the pairs
$(y_1+y_2+\cdots+y_i+y_{i+2}, y_1+y_2+\cdots+y_{i+3})$ and
$(y_1+y_2+\cdots+y_{i+2}, y_1+y_2+\cdots+y_{i+3})$. As we argued above,
the position of the last digit of $y_1+y_2+\cdots+y_i$ is the
position of the last digit of $y_i$. This means that
$y_1+y_2+\cdots+y_i$, $y_{i+1}$ and $y_{i+2}$ have right to left
disjoint supports, thus this case is analogous to Case 1.
\end{enumerate}
Therefore, we cannot have 3 consecutive terms with right to left
disjoint supports.
\end{proof}
\begin{claim1} By passing to a block subsequence, we may assume that the
sequence $(y_n)_{n\geq 1}$ contains no two consecutive terms with right
to left disjoint supports.
\end{claim1}
\begin{proof}
Using Claim 4, we construct the new block subsequence $(z_n)_{n\geq1}$
inductively, with each term being either a $y_i$ or a sum of two
consecutive $y_i$. If $y_1$ and $y_2$ do not have right to left disjoint
supports, we do not change them. If they do, then $y_2$ and $y_3$ do not
have right to left disjoint supports. We then replace $y_1$ by $y_1+y_2$
and relabel the sequence. Thus now the first and the second terms do not
have right to left disjoint supports.\\Assume we have built our block
subsequence up to the $i^{\text{th}}$ term: thus we have
$(z_n)_{n=1}^i$, with $z_i$ being the sum of at most two consecutive
terms of our original sequence. Thus $z_i=y_{k}+y_{k+1}$ or
$z_i=y_{k+1}$, for some $k\in\mathbb N$. If $y_{k+1}$ and $y_{k+2}$ do
not have right to left disjoint supports, then we let $z_{i+1}=y_{k+2}$.
If on the other hand $y_{k+1}$ and $y_{k+2}$ do have right to left
disjoint supports, then $y_{k+2}$ and $y_{k+3}$ do not have right to
left disjoint supports, so we replace $z_i$ by $z_i+y_{k+2}$ and set
$z_{i+1}=y_{k+3}$. We note that since $y_{k+1}$ and $y_{k+2}$ have right
to left disjoint supports, by Claim 4 this implies that $y_k$ and
$y_{k+1}$ do not have right to left disjoint supports, thus, by our
inductive construction we must have $z_i=y_{k+1}$. Hence, when we
perform the inductive step in this case, $z_i$ will be replaced by
$y_{k+1}+y_{k+2}$, a sum of two consective terms of our original
sequence. This gives us our block subsequence up to the
$(i+1)^{th}$ term.
\end{proof}
Note that Claim 5 is true for any block subsequence of $(y_n)_{n\geq 1}$
as well. This is an immediate consequence of Claim 2 and the fact that
Claim 2 is preserved by passing to any block subsequence.\\
\par For a natural number $n$ we denote the positions of its last and
first digits by $l_n$ and $f_n$, respectively. Let $a<b$ with $f_a<f_b$,
$l_a<l_b$, and $a$ and $b$ having disjoint supports, but not right to
left disjoint supports. We define a \textit{fragment of b in a} to be a
maximal binary string in $a+b$ that appears in $b$ at the same
positions, has the digit 1 at its last position, and is situated between
$l_a$ and $f_b$ inclusive. More formally, let
$a+b=r_{k_1}r_{k_1-1}\cdots r_1$, $b=b_{k_2}b_{k_2-1}\cdots b_1$  and
$a=a_{k_3}a_{k_3-1}\cdots a_1$ be the binary representations of $a+b$,
$b$ and $a$. A binary string $s_ks_{k-1}\cdots s_1$ is called a
\textit{fragment of $b$ in $a$} if there exists a positive integer $t$
such that $k+t-1\leq l_a$, $t\geq f_b$, $r_{k+t-1}r_{k+t-2}\cdots
r_t=b_{k+t-1}b_{k+t-2}\cdots
b_t=s_ks_{k-1}\cdots s_1$, $r_{k+t-1}=b_{k+t-1}=1$, $r_{t-1}\neq
b_{t-1}$ or $t=f_b$, and there exists no binary string $w_d\cdots w_1$
with $w_d=1$ such that $w_d\cdots w_1=b_{k+t+d-1}\cdots
b_{k+t}=r_{k+t+d-1}\cdots r_{k+t}$ and $k+t+d-1\leq l_a$. We sometimes
refer to these fragments as the \textit{right
fragments of $b$ in $a$}. Similarly, a binary string $s_ps_{p-1}\cdots
s_1$ is called a \textit{fragment of $a$ in $b$} if there exists a
positive integer $l$ such that $p+l-1\leq l_a$, $l\geq f_b$,
$r_{p+l-1}r_{p+l-2}\cdots
r_l=a_{p+l-1}a_{p+l-2}\cdots a_l=s_ps_{p-1}\cdots s_1$,
$r_{p+l-1}=a_{p+l-1}=1$, $r_{l-1}\neq a_{l-1}$ or $l=f_b$, and there
exists no binary string $v_e\cdots v_1$ with $v_e=1$ such that
$v_e\cdots v_1=a_{p+l+e-1}\cdots a_{p+l}=r_{p+l+e-1}\cdots r_{p+l}$ and
$p+l+e-1\leq l_a$. We sometimes refer to
these fragments as the \textit{left fragments of $a$ in $b$}. Note that 
there is always at least one left fragment of $a$ in $b$ and at least 
one right fragment of $b$ in $a$, because $a$ and $b$ have disjoint 
supports but \textit{not} right to left disjoint supports. The picture below illustrates this definition in the case where there is only one fragment.
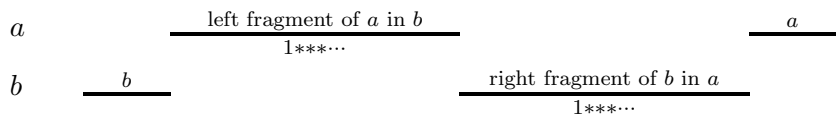
\begin{alignat*}{9}
& a\hspace{2em} & & \overset{\text{left fragment of
}a\text{ in }b}{\underset{1***\cdots}{\rule{10em}{1pt}}} & & \overset{a}{\rule{3em}{1pt}} \\
& b & \overset{b}{\rule{3em}{1pt}}& & \overset{\text{right fragment of
}b\text{ in }a}{\underset{1***\cdots}{\rule{10em}{1pt}}} & 
\end{alignat*}

\par Now
let $a<b<c$ with the property that $l_a<l_b<l_c$, $f_a<f_b<f_c$,
$l_a+1<f_c$, and such that they have disjoint supports, but the pairs
$(a,b)$ and $(b,c)$ do not have right to left disjoint supports. The
\textit{fragments of $b$
with respect to $a$ and $c$} are the fragments of $b$ in $a$ together
with the fragments of $b$ in $c$. Whenever we count fragments, we count
them with multiplicity -- so for example, if the string 10110 occurs as a
fragment in two different places, then we count this as two
fragments. Note that fragments do not overlap by the maximality
condition. \par Let $p<r<s$ be three natural numbers with $f_p<f_r<f_s$,
$l_p<l_r<l_s$, $l_p\geq f_r$, $l_r\geq f_s$ and $l_p+1<f_s$.
We define \textit{the centre of r with respect to $p$ and $s$} to be the
binary string in $r$ situated strictly between $l_p$ and $f_s$. We note
that the centre of $r$ cannot be the empty string, although, unlike a
fragment in the disjoint case, it can certainly be a string of `0's.
\par The picture below
illustrates the concepts we have just defined. The fragments are with
respect to the three numbers $a$, $b$ and $c$. For example, the centre
of $b$ is the centre of $b$ with respect to $a$ and $c$. Here there is
only one left fragment of $a$ and only one right fragment of $b$; in
general, of course, there could be several, alternating from one to the
other.
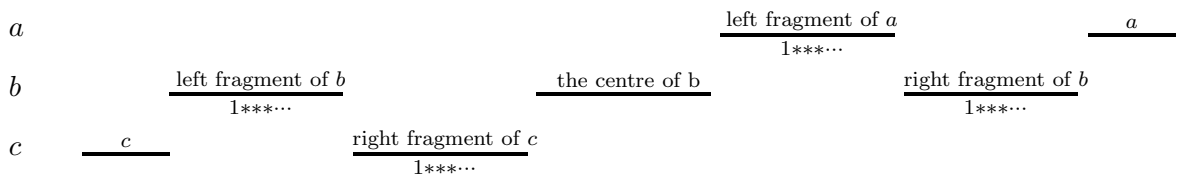
\begin{alignat*}{9}
& a\hspace{2em} &&  &&  && &&  && \overset{\text{left fragment of
}a}{\underset{1***\cdots}{\l}} && && \overset{a}{\rule{3em}{1pt}} \\
& b &&  && \overset{\text{left fragment of
}b}{\underset{1***\cdots}{\l}} &&  && \overset{\text{the centre of
b}}{\l} &&  && \overset{\text{right fragment of
}b}{\underset{1***\cdots}{\l}} && && \\
& c && \overset{c}{\rule{3em}{1pt}} &&  &&\overset{\text{right fragment
of }c}{\underset{1***\cdots}{\l}}  && && && &&
\end{alignat*}\par When working with a sequence $(y_n)_{n\geq 1}$, we
consider the fragments or the centre of a term or of a consecutive sum
of terms to be with respect to its neighbours. In other words, for any
$1<i<j$, the fragments and centre of $y_i+y_{i+1}+\cdots+y_{j-1}$ are
implicitly understood to be with respect to $y_{i-1}$ and $y_{j}$.\par
Let $m$ and $n$ be two natural numbers such that $l_m<l_n$ and
$f_m<f_n$. For each $i$, let $m_i$, $n_i$ and $(m+n)_i$ be the digits of
$m$, $n$ and $m+n$ at position $i$, respectively. When adding $m$ and
$n$ in binary, it is convenient to refer to the minimal interval in
which all binary carrying occur as the \textit{carry region} or just the
\textit{carry}. More precisely, the carry region \textit{starts} at the
least $i$ for which $m_i=n_i=1$, and \textit{stops} at position $k$,
where $k$ is the maximum $i$ such that $(m+n)_i\neq m_i+n_i$. For
example, if $m$ is 100111010010 and $n$ is 1010011011100, then the carry
starts at position 4 and stops at position 9.
\begin{claim1} There exists no $i\geq1$ with the following property: $y_i$, $y_{i+1}$,
$y_{i+2}$, $y_{i+3}$ and $y_{i+4}$ have pairwise disjoint supports and
each of the centres of $y_{i+1}$, $y_{i+2}$, $y_{i+3}$, $y_{i+4}$,
$y_{i+1}+y_{i+2}$ and $y_{i+2}+y_{i+3}$ are a string of `1's.\end{claim1}
\begin{proof}
Suppose for a contradiction that such an $i$ exists. Let $y_{i+1}$ have
$k_1$ intervals (i.e.~$k_1$ disjoint strings of `1's) between the position of the last digit of $y_i$ and the
position of its first digit inclusive, and $k_2$ intervals between the
position of its last digit and the position of the first digit of
$y_{i+2}$ inclusive. Because we assumed the centre of $y_{i+1}+y_{i+2}$
is a string of `1's, we get that $y_{i+1}$ and $y_{i+2}$ complement each
other between the position of the first digit of $y_{i+2}$ and the
position of the last digit of $y_{i+1}$ inclusive. Therefore $y_{i+2}$
has $k_2$ intervals between these 2 positions too. Similarly, if
$y_{i+2}$ has $k_3$ intervals between the position of its last digit and
the position of the first digit of $y_{i+3}$, then so does $y_{i+3}$.
Finally, let $y_{i+3}$ have $k_4$ intervals between the position of its
last digit and the position of the first digit of $y_{i+4}$
inclusive.\par The reader might find the diagram below helpful, where the two dotted  fragments are intervals as a result of $y_{i+1}$ and $y_{i+2}$ complementing each other in order to have an interval as the centre of $y_{i+1}+y_{i+2}$. In the example below we have $k_2=1$, and only one right fragment of $y_{i+1}$ in $y_i$ that contains $k_1$ fragments. The number $k_1$ does not depend on the number of such fragments: it is the sum of the number of intervals in the fragments.
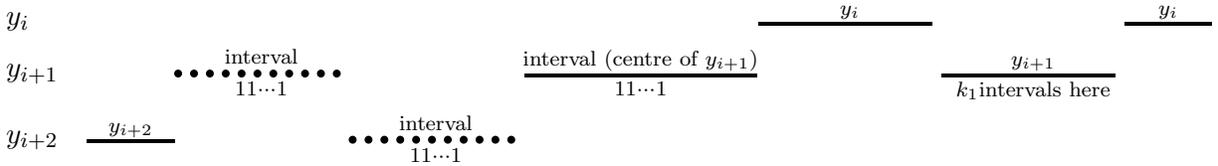
\begin{alignat*}{9}
& y_i\hspace{2em} &&  &&  && &&  && \overset{y_i}{\l} && && \overset{y_i}{\rule{3em}{1pt}} \\
& y_{i+1} &&  && \overset{\text{interval}}{\underset{11\cdots1}{\dhorline{6em}{3pt}}} &&  && \overset{\text{interval (centre of }y_{i+1})
}{\underset{11\cdots1}{\rule{8em}{1pt}}} &&  && \overset{y_{i+1}}{\underset{k_1 \text{intervals here}}{\l}} && && \\
& y_{i+2} && \overset{y_{i+2}}{\rule{3em}{1pt}} &&  &&\overset{\text{interval}}{\underset{11\cdots1}{\dhorline{6em}{3pt}}}&& && && &&\end{alignat*}Since each centre is an interval and all numbers have
disjoint supports, we get that $y_{i+1}$ has $1+k_1+k_2$ intervals,
$y_{i+2}$ has $1+k_2+k_3$ intervals, $y_{i+3}$ has $1+k_3+k_4$
intervals,
$y_{i+1}+y_{i+2}$ has $1+k_1+k_3$ intervals, and $y_{i+1}+y_{i+3}$ has
$k_1+k_2+k_3+k_4+2$ intervals since $y_{i+1}$ and $y_{i+3}$ have
disjoint right to left supports that are at least one position
apart.\\By looking at the $I$ value of these numbers, $c_5$ tells us
that\begin{center}$1+k_1+k_2\equiv
1+k_2+k_3\equiv1+k_3+k_4\equiv1+k_1+k_3\equiv k_1+k_2+k_3+k_4+2 \mod
2.$\end{center}
The first four equations imply that $k_1$, $k_2$, $k_3$ and $k_4$ have
the same parity. Hence $k_1+k_2+k_3+k_4+2$ is even, which implies that
$k_1+k_2+1$ is even, a contradiction.
\end{proof}
\par It is important to note that Claim 6 implies that our sequence
$(y_n)_{n\geq
1}$, and thus any of its block subsequences, cannot be of Type B.
Indeed, if the sequence $(y_n)_{n\geq 1}$ is of Type B, then so are all of
its block subsequences, and so the first
digit of $y_1+y_2+\cdots+y_k$ is at the same position as the first digit
of $y_{k+1}$ for all $k\geq 1$. If we first look at $y_1$, $y_2$ and
$y_3$, we notice that the above conditions imply that the centre of
$y_2$ has to be an interval, otherwise the carry in $y_1+y_2$ would stop
before the position of the first digit of $y_3$. Moreover, if we look at
the block subsequence obtained by just replacing $y_2$ with $y_2+y_3$,
we must also have that the centre of $y_2+y_3$ is an interval. This
immediately implies that $y_2$ and $y_3$ must have disjoint supports,
otherwise the first position they both have a `1' at will become a `0'
in $y_2+y_3$, as well as being part of the centre.\par Recapping, we have shown that if $(y_n)_{n\geq 1}$ is of Type B then the centre of $y_2$ (with respect to $y_1$ and $y_3$) is an interval, the centre of $y_2+y_3$ (with respect to $y_1$ and $y_4$) is an interval, and $y_2$ and $y_3$ have disjoint supports. Passing to the block subsequence $y_1+y_2, y_3, y_4, \cdots$ and repeating the argument, we find that the centre of $y_3$ (with respect to $y_1+y_2$ and $y_4$) is an interval, the centre of $y_3+y_4$ (with respect to $y_1+y_2$ and $y_5$) is an interval, and $y_3$ and $y_4$ have disjoint supports. By Claim 3, the position of the last digit of $y_1+y_2$ is the same as that of $y_2$, so the centre of $y_3$ with respect to $y_1+y_2$ and $y_4$ is the same as the centre of $y_3$ (with respect to $y_2$ and $y_4$), and similarly for $y_3+y_4$. Continuing
inductively, we obtain that for all $n\geq 2$ the centres of $y_n$ and $y_n+y_{n+1}$ are intervals, and the terms $y_n$ and $y_{n+1}$ have disjoint supports, which contradicts Claim 6.\par Therefore we can guarantees
that in what follows all sequences are of Type A.
\begin{claim1} There exists no $i\in\mathbb N$ such that $y_i$,
$y_{i+1}$, $y_{i+2}$, \ldots, $y_{i+15}$ have pairwise disjoint
supports.
\end{claim1}
\begin{proof} Suppose for a contradiction that such an $i$ exists. We
will find a block subsequence of $(y_n)_{n\geq 1}$ that will not satisfy
the conditions in Theorem 4, a contradiction. By Claim 2 we know that if
three consecutive terms have disjoint supports, then the positions
between the first and the last digit of their sum inclusive can be
partitioned into fragments such that each fragment corresponds to
exactly one term $y_i$, as illustrated below.
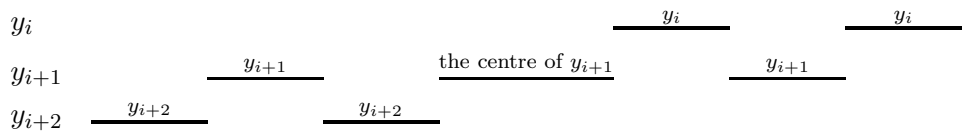
\begin{alignat*}{9}
& y_i\hspace{2em} &&  &&  && &&  && \overset{y_i}{\rule{4em}{1pt}} && &&
\overset{y_i}{\rule{4em}{1pt}} \\
& y_{i+1} &&  && \overset{y_{i+1}}{\rule{4em}{1pt}} &&  &&
\overset{\text{the centre of } y_{i+1}}{\rule{6em}{1pt}} &&  &&
\overset{y_{i+1}}{\rule{4em}{1pt}} && && \\
& y_{i+2} && \overset{y_{i+2}}{\rule{4em}{1pt}} &&
&&\overset{y_{i+2}}{\rule{4em}{1pt}}  && && && &&
\end{alignat*}
We immediately observe that every fragment in the picture, except for
the centre of $y_{i+1}$, has to contain the digit 1, by definition of
fragments.\\
As we noted above, the centres can be strings of `0'. However, since the
last
digit of $y_{i+1}$ is contained in the centre of $y_{i+1}+y_{i+2}$ that
sits between $y_i$ and $y_{i+3}$, we can replace $y_{i+1}$ with
$y_{i+1}+y_{i+2}$, $y_{i+2}$ with $y_{i+3}+y_{y+4}$, $y_{i+3}$ with
$y_{i+5}+y_{i+6}$, \ldots, $y_{i+7}$ with $y_{i+13}+y_{i+14}$ and
relabel the sequence. Thus, by passing to a block subsequence, we may
assume that we can find 9 consecutive terms, $y_k$, $y_{k+1}$,
$y_{k+2}$, $y_{k+3}$, $y_{k+4}$, \ldots, $y_{k+8}$, such that they have
disjoint supports and the centre of $y_{k+1}$, $y_{k+2}$, $\ldots$,
$y_{k+7}$ all contain the digit 1.\\
The next step is to look at what happens with the sum
$y_1+y_2+\cdots+y_k$. We know that, by disjointness, at the position of
the first digit of $y_{k+1}$, $y_k$ has a `0'. If the centre of
$y_k$ contains at
least one `0', or $k=1$, then the sum $y_1+y_2+\cdots+y_k$ and $y_{k+1}$
have the
same fragment interaction as $y_k$ and $y_{k+1}$ (in other words, the
fragments of $y_k$ in $y_{k+1}$ are the same as the fragments of
$y_1+y_2+\cdots+y_k$ in $y_{k+1}$, and the fragments of $y_{k+1}$ in
$y_{k}$ are the same as the fragments of $y_{k+1}$ in
$y_1+y_2+\cdots+y_k$) since the carry stops before the fragments start,
and when $k=1$ there is no carry to consider as the above sum is just
$y_1$.
Here we used the fact that the last digit of $y_1+y_2+\cdots+y_{k-1}$ is at the same position as the last digit of $y_{k-1}$ for $k\geq 2$.\\However, Claim 6 tells us that amongst 5 consecutive terms with disjoint supports, we can always find one or a sum of two consecutive terms that does not have the centre a string of
`1's (since Claim 6 is invariant under taking block subsequences). Therefore, by passing to a block subsequence or ignoring some
previous terms, we can assume that the centre of $y_k$ is not an
interval, or $k=1$.\\ Finally, by passing to a block subsequence, we may
assume
that we can find 5 consecutive terms $y_t$, $y_{t+1}$, $y_{t+2}$,
$y_{t+3}$ and $y_{t+4}$ such that the centres of $y_{t+1}$, $y_{t+2}$
and $y_{t+3}$ each contain at least one `1', $y_1+y_2+\cdots+y_t$
interacts with the fragments of $y_{t+1}$ the same way $y_t$ does, and
all 5 terms have pairwise disjoint supports.
\\ We now look at the value of $J$ for the following pairs:
$(y_1+\cdots+y_t+y_{t+3}, y_1+y_2+\cdots+ y_{t+4})$,
$(y_1+\cdots+y_t+y_{t+2}+y_{t+3}, y_1+y_2+\cdots+y_{t+4})$,
$(y_1+\cdots+y_t+y_{t+1}+y_{t+3},y_1+y_2+\cdots+y_{t+4})$ and
$(y_1+\cdots+y_{t+3}, y_1+\cdots+y_{t+4})$. Let $y_{t+1}$ have $l_{t+1}$
fragments on its left and $r_{t+1}$ fragments on its right. We define
$r_{t+2}$, $r_{t+3}$, $l_{t+2}$ and $l_{t+3}$ similarly. We notice that
$y_{t+1}+y_{t+2}$ has $l_{t+2}$ fragments on its left and $r_{t+1}$
fragments on its right. We also notice, by the definition of fragments,
that $r_{t+2}=l_{t+1}$. If we look at the first pair above, the term
$y_{t+1}+y_{t+2}$ is missing from the first sum. So the non-zero digits in its fragments will
all be labelled `1'. Therefore, its right fragments will give $r_{t+1}+1$
jumps, while its left fragments will give $l_{t+2}$ jumps. Hence, the
missing term gives $r_{t+1}+l_{t+2}+1$ jumps. Similarly for the next two
pairs, the missing terms give $r_{t+1}+l_{t+1}+1$ and
$r_{t+2}+l_{t+2}+1$ jumps, respectively. For the last pair there is no
missing term, so the jumps come from the interaction between $y_{t+3}$
and $y_{t+4}$, which is identical for the other three pairs by
disjointness. All the other digits in all four pairs remain
unchanged.\par The explanation above is summarised as follows:
\begin{center}
$J(y_1+\cdots+y_t+y_{t+3}, y_1+\cdots+y_{t+4})-J(y_1+\cdots+ y_{t+3}, y_1+\cdots +y_{t+4})=r_{t+1}+l_{t+2}+1$,\\
$J(y_1+\cdots+ y_t+y_{t+2}+y_{t+3}, y_1+\cdots+y_{t+4})-J(y_1+\cdots+y_{t+3}, y_1+\cdots+y_{t+4})=r_{t+1}+l_{t+1}+1$,\\
$J(y_1+\cdots y_t+y_{t+1}+y_{t+3}, y_1+\cdots +y_{t+4})-J(y_1+\cdots+y_{t+3}, y_1+\cdots+y_{t+4})=r_{t+2}+l_{t+2}+1$.
\end{center}
Since our coloring asks for the $J$ values to have same parity, we need $0\equiv r_{t+1}+l_{t+2}+1\equiv
r_{t+1}+l_{t+1}+1\equiv r_{t+2}+l_{t+2}+1\mod 2$.
Because $r_{t+2}=l_{t+1}$, the last equation tells us that $l_{t+2}$ and
$l_{t+1}$ have different parities. However, by taking the difference of
the first two equations, we must have that they have the same parity, a
contradiction.\end{proof}
\begin{claim1} By passing to a block subsequence, we may assume that the
sequence $(y_n)_{n\geq 1}$ contains no two consecutive terms with
disjoint supports.\end{claim1}
\begin{proof}
The same as the proof of Claim 5.\end{proof}
\begin{claim1} By passing to a block subsequence, we may assume that for
every $n\geq 1$ the carry in any sum where the biggest term is $y_n$,
stops before the position of the first digit of $y_{n+1}$.
\end{claim1}
\begin{proof}
As in Claim 7, it is enough to show that for every $n\geq 2$, the centre
of every $y_n$ contains at least one `0'. We will prove
this by induction, replacing terms by consecutive sums and relabelling,
and also bearing in mind that our initial sequence does not have any two
consecutive terms with disjoint supports. Assume we have built the
sequence with the desired property up to the $n^{\text{th}}$ term. The
terms $y_{n+1}$ and $y_{n+2}$ are consecutive terms of the original
sequence, so their supports are not disjoint. If the centre of
$y_{n+1}$ contains a `0', then we have found the ${(n+1)}^{\text{th}}$
term. If it does not contain a `0', then we take $y_{n+1}+y_{n+2}$ to be
the ${(n+1)}^{\text{th}}$ term. To see that this satisfies the claim, we
notice that since $y_{n+1}$ and $y_{n+2}$ do not have disjoint supports,
the first position at which both have a `1', becomes a `0' in
$y_{n+1}+y_{n+2}$. As the sequence $(y_n)_{n \geq 1}$ is of Type A, we see that that position is part
of the centre of $y_{n+1}+y_{n+2
}$. Note that the base case $n=2$  is the same as the induction step.
Thus the claim is proved.
\end{proof}
Note that the condition in Claim 9 is invariant under passing to
a block subsequence.

\par We also note that the property that no two consecutive terms have
disjoint supports is not necessarily preserved by passing to a block
subsequence. We also observe that we have altered the sequence in Claim
8 that was assumed not to have two consecutive terms with disjoint
supports, and obtained one such that the carry of any sum with biggest
term $y_n$ stops before the support of $y_{n+1}$ begins. Further, this
property is preserved by passing to a block subsequence. Therefore,
starting with a sequence $(y_n)_{n\geq 1}$ with this property, we can
repeat the process in Claim 7 and Claim 8 again and assume that
$(y_n)_{n\geq 1}$ has both the property that the binary carry of any sum
stops before the support of the next term starts, and also the property
that no two consecutive terms have disjoint supports. These two
properties together are invariant under our standard operation of
passing to a block subsequence (noting that the property of `consecutive terms do not have disjoint supports' is preserved because the carry resulting from any
earlier additions is guaranteed to stop
before the supports overlap).\\\par For a sequence $(z_n)_{n\geq 1}$
that is of Type A and has the two properties we
have stated in the previous paragraph, we define $j_n$, for $n\geq 2$, to be the maximum of
the position of where the carry of $z_n+z_{n-1}$ stops (or equivalently
any finite sum of the $z_i$ with greatest terms $z_n$ and $z_{n-1}$) and
the position of
the last digit of $z_{n-1}$. For completeness, we set $j_1$ to
be one less than the position of the first digit of $y_1$. We
also define the \textit{middle} of $z_n$ to be the (possible empty) binary string 
contained strictly between $j_n$ and the position
of the first digit of $z_{n+1}$. We call the middle of $z_n$ \textit{proper} if it is nonempty and it contains at least one nonzero digit. Finally, we define the \textit{overlapping zone} of $z_n$ and
$z_{n+1}$ to be the consecutive set of positions between the position of
the first digit of $z_{n+1}$ and $j_{n+1}$ inclusive.

\begin{claim1} By passing to a block subsequence, we may assume that the
middle of $y_n$ is proper for all $n\geq 2$.
\end{claim1}
\begin{proof} We prove the claim by induction. Assume that all terms up
to $y_{n-1}$, $n\geq 3$, have a proper middle. If $y_n$ has a proper
middle, then we move on to the next term. If $y_n$ does not have a
proper middle, then $y_n+y_{n+1}$ has a proper middle with respect to
$y_{n-1}$ and $y_{n+2}$. This is because at position $j_{n+1}$ 
in the sum $y_n+y_{n+1}$
we find
the digit 1 by definition. Note that by Claim 9 the `new $j_n$' (corresponding to 
$y_n+y_{n+1}$) is equal to the `old $j_n$' (corresponding to $y_n$). Also, 
$j_{n+1}$ is less than the position of
the first digit of $y_{n+2}$ and, by construction,
$j_{n+1}>j_n$. Thus $y_n+y_{n+1}$ does have a proper middle. Therefore
we take the $n^{\text{th}}$ term to be $y_{n}+y_{n+1}$, and relabel the
rest of the sequence, thus complete the induction step. We note that the
same argument directly gives that the middle of $y_2$ can be assumed to
be proper, which finishes the proof.
\end{proof}

Note that, given that a sequence satisfies the conditions of Claim 9, the
conditions in Claim 10 are invariant under taking block subsequences. By
earlier remarks, we may now therefore assume that out sequence satisfies
Claim 8, Claim 9 and Claim 10.
\\
\par \textit{Stage 3. }We now add a final piece of notation. For positive integers $a$ and $b$, that do not have disjoint supports,
consider the positions where binary carries
occur in the sum $a+b$. Those positions form some intervals which we call the \textit{carry intervals} of $a$ and $b$. For example, if $a=110100010$ and $b=10100111$, then the carry intervals are $\{1,2,3 \}$, $\{5,6\}$ and $\{7,8,9\}$.
\par Let $m<n$ be two positive integers such that $m$ and $n-m$ do not have disjoint supports. We label a position by `2' if it is not part of any carry interval of $m$ and $n-m$, and both $m$ and $n$ have the digit 1 at that position. Also, we label a position by `1' if it is not part of any carry interval of $m$ and $n-m$, and exactly one of $m$ and $n$ has a nonzero digit at that position. Let $\tilde{J}(m,n)$ be the number of jumps from a position labelled `2' to a position labelled `1', as we read the labels from right to left (ignoring the positions that do not have labels).\par Returning to our sequence, let $y_n'$ be the number obtained from $y_n$ by changing all the digits in the carry intervals of $y_n$ and $y_{n-1}$, and in the carry intervals of $y_n$ and $y_{n+1}$, to 0, for each $n>1$. Let also $y_1'$ be the number obtained from $y_1$ by changing all the digits in the carry interval of $y_1$ and $y_2$ to 0. Note that the new sequence $(y_n')_{n \geq 1}$ is still increasing and of Type A as a consequence of Claim 10, and that its terms have pairwise disjoint supports.
\par With this in mind, our final colouring is: we colour $(a, b)$ by $(c_0, c_1, c_2, c_3, c_4, c_5, c_6)$, where $c_0$, $c_1$, $c_2$, $c_3$, $c_4$, $c_5$ are defined above, and $c_6=\tilde{J}(a,b)\mod 2$, with $c_6=0$ or 1, if $a$ and $b-a$ do not have disjoint supports, and $c_6=3$ if $a$ and $b-a$ have disjoint supports.
\par Let $(y_1+y_{k_1}+\cdots+y_{k_t}, y_1+\cdots+y_{k_{t+1}})$ be any of the pairs that have the same colour. We first observe that, by Claim 8 and Claim 9, $y_1+y_{k_1}+\cdots+y_{k_t}$ and $y_1+\cdots+y_{k_{t+1}}-(y_1+y_{k_1}+\cdots+y_{k_t})=y_2+\cdots+y_{k_1-1}+\cdots+y_{k_t+1}+\cdots+y_{k_{t+1}}$ never have disjoint supports -- for example, $y_{k_{t}}$ and $y_{k_t+1}$ do not have disjoint supports and, in the above sums, they are unchanged in their overlapping zone. We therefore have that $c_6\neq3$. Moreover, $\tilde{J}(y_1+y_{k_1}+\cdots+y_{k_t}, y_1+\cdots+y_{k_{t+1}})=J(y_1'+y_{k_1}'+\cdots+y_{k_t}', y_1'+\cdots+y_{k_{t+1}}')$, and so the same argument as in Claim 7 gives us a contradiction. This completes the proof of Theorem 4.
\end{proof}
\section{Open Problems}
The colouring of $\mathbb N^{(2)}$ above, constructed in the previous
section, involves
colouring pairs. But can Theorem 4 be solved by a colouring that comes
in a natural
way just from a colouring of numbers? In particular, what happens if we
promise
that our colouring for Theorem 4 gives $(a,b)$ a colour that depends
only on the value
of $a+b$?

In this case, the sum $a+b$, for a pair $(a,b)$ as in the statement of
Theorem 4,
is exactly a sum $a_1 y_1 + a_2 y_2 + \cdots + a_k y_k$, where each
$a_i$ is 1 or 2
with $a_k=1$ and $a_1=2$. Replacing $y_1$ with $2 y_1$, this yields the
following
question.

\begin{question} Is it true that whenever $\mathbb N$ is finitely
coloured, there exists a
sequence $(y_n)_{n\geq 1}$ such that every sum $a_1 y_1 + a_2 y_2 +
\cdots + a_k y_k$,
where each $a_i$ is 1 or 2, with $a_1=a_k=1$, has the same
colour?\end{question}

In general, such Ramsey-type statements, in which each coefficient can
vary
independently between some values, tend to be false. But here the fact
that there are
no `gaps', in other words that the $y_i$ in a given sum form an initial
segment of
the sequence $(y_n)_{n\geq 1}$, seems to perhaps make a difference.

We mention that if one allows $a_k$ to be 1 or 2, then the result is
easily seen to
be false, because one sum will be forced to be roughly double another,
which can be
ruled out by a suitable colouring. And if one instead allows $a_1$ to be
1 or 2 then
the result is also false, by considering the 2-colouring given by the
least
significant non-zero digit in the base 3 expansion of a number. Finally,
if one allows
`gaps', so that some of the $a_i$ are allowed to be zero, then it turns
out that the
result is again false, by using a colouring that examines the lengths of
the
jumps between successive elements of the support of a number: this is
similar to the colourings considered in \cite{Deuber}.

It is possible that Question 5 might be related to a problem considered
by Hindman,
Leader and Strauss \cite{Imre}. They conjectured that whenever $\mathbb
N$ is finitely coloured there
exists a sequence $(y_n)_{n\geq 1}$ such that all finite sums of the
$y_i$, and also all sums
of the form $y_{n-1} + 2 y_n + y_{n+1}$, are the same colour. In each of
these problems, it is the fact that the terms must be consecutive (in
each sum for Question 5, and for the sums
$y_{n-1} + 2 y_n + y_{n+1}$ in the conjecture of Hindman, Leader and
Strauss) that
causes the difficulty. We mention that if one attempts to strengthen the
conjecture of Hindman, Leader and Strauss in almost any significant way then the
resulting statement turns out to be false: this is related to the `inconsistency'
of Milliken-Taylor systems (see~\cite{Deuber} and the discussion in~\cite{Imre}). 
\par
Finally, returning to infinite words, what happens in Theorem 1 if we
relax
the condition that the factors $u_n$ form an actual factorisation
of our word $x$:
what if we allow some gaps between them? Could it be that we can
actually
allow gaps, as long as they are bounded, and still find a bad colouring?
This is a natural question to ask, in light of some variants of
Hindman's theorem, such as Theorem 5.23 of \cite{HS}.

\begin{question} Let $x$ be an infinite word on alphabet $X$ that is not
eventually periodic. Must there exist a finite colouring of $X^*$ such
that
there does not exist a sequence $u_1,u_2,\cdots$ of factors of $x$, with
$0 \leq A_x(u_{n+1}) - B_x(u_n) \leq C$ for all $n$ (for some $C$), such
that all the words
$u_{k_1} u_{k_2} \cdots u_{k_n}$, where $k_1 < k_2 \cdots < k_n$, have
the
same colour?
\end{question}

Note that if we insist that $C=0$ then this is precisely Theorem 1.\\

\textbf{Acknowledgement.} We would like to thank the referee for their very careful reading of the paper: their comments and suggestions have greatly improved the clarity of the paper.

\bibliographystyle{amsplain}
\bibliography{document}
\Addresses
\end{document}